\documentclass[11pt]{amsart}
\usepackage{amssymb,amsmath,amsfonts,latexsym, mathtools}
\usepackage{bm, enumerate}
\usepackage{graphicx}
\usepackage{tikz}
\usepackage{tikz-cd}
\usepackage{hyperref}
\newcommand{\newsection}[1]{\setcounter{equation}{0} \section{#1}}
\newcommand{\bea}{\begin{eqnarray}}
\newcommand{\eea}{\end{eqnarray}}

\newcommand{\clb}{\mathcal{B}}

\newcommand{\cld}{\mathcal{D}}
\newcommand{\cle}{\mathcal{E}}

\newcommand{\clh}{\mathcal{H}}
\newcommand{\clk}{\mathcal{K}}

\newcommand{\clj}{\mathcal{J}}
\newcommand{\cll}{\mathcal{L}}
\newcommand{\clm}{\mathcal{M}}
\newcommand{\cln}{\mathcal{N}}

\newcommand{\clp}{\mathcal{P}}
\newcommand{\clq}{\mathcal{Q}}
\newcommand{\clr}{\mathcal{R}}
\newcommand{\cls}{\mathcal{S}}

\newcommand{\clu}{\mathcal{U}}

\newcommand{\T}{\mathcal{T}}
\newcommand{\sq}{\mathsf{Q}}
\newcommand{\sr}{\mathsf{R}}
\newcommand{\s}{\mathsf{S}}
\newcommand{\st}{\mathsf{T}}
\newcommand{\sv}{\mathsf{V}}
\newcommand{\sw}{\mathsf{W}}
\newcommand{\sa}{\mathsf{A}}
\newcommand{\sanb}{\mathsf{B}}
\newcommand{\sanc}{\mathsf{C}}

\newcommand{\D}{\mathbb{D}}

\newcommand{\Z}{\mathbb{Z}}
\newcommand{\z}{\bm{z}}

\def\textmatrix#1&#2\\#3&#4\\{\bigl({#1 \atop #3}\ {#2 \atop #4}\bigr)}
\def\dispmatrix#1&#2\\#3&#4\\{\left({#1 \atop #3}\ {#2 \atop #4}\right)}
\newcommand{\be}{\begin{equation}}
\newcommand{\ee}{\end{equation}}
\newcommand{\ben}{\begin{eqnarray*}}
\newcommand{\een}{\end{eqnarray*}}

\newcommand{\bi}{\begin{itemize}}
\newcommand{\ei}{\end{itemize}}

\newtheorem{Theorem}{\sc Theorem}[section]
\newtheorem{Lemma}[Theorem]{\sc Lemma}
\newtheorem{Proposition}[Theorem]{\sc Proposition}
\newtheorem{Corollary}[Theorem]{\sc Corollary}
\newtheorem{Definition}[Theorem]{\sc Definition}
\newtheorem{Example}[Theorem]{\sc Example}
\newtheorem{Remark}[Theorem]{\sc Remark}
\newtheorem{Remarks}[Theorem]{\sc Remarks}
\newtheorem{Note}[Theorem]{\sc Note}
\newtheorem{Question}{\sc Question}
\newtheorem{ass}[Theorem]{\sc Assumption}
\newcommand{\bt}{\begin{Theorem}}
\def\beginlem{\begin{Lemma}}
\def\beginprop{\begin{Proposition}}
\def\begincor{\begin{Corollary}}
\def\begindef{\begin{Definition}}
\def\beginexamp{\begin{Example}}
\def\beginrem{\begin{Remark}}
\def\beginq{\begin{Question}}
\def\beginass{\begin{ass}}
\def\beginnote{\begin{Note}}
\newcommand{\et}{\end{Theorem}}
\def\endlem{\end{Lemma}}
\def\endprop{\end{Proposition}}
\def\endcor{\end{Corollary}}
\def\enddef{\end{Definition}}
\def\endexamp{\end{Example}}
\def\endrem{\end{Remark}}
\def\endq{\end{Question}}
\def\endass{\end{ass}}
\def\endnote{\end{Note}}

\begin{document}
\title{Factorization of Toeplitz operators}

\author[Panja]{Samir Panja}
\address{Department of Mathematics, Indian Institute of Technology Bombay, Powai, Mumbai, 400076, India}
\email{panjasamir2020@gmail.com, spanja@math.iitb.ac.in}


\begin{abstract}
In this article, by considering $T=(T_1,\dots, T_n)$, an $n$-tuple of commuting contractions  on a Hilbert space $\clh$, we study $T$-Toeplitz operators which consists of bounded operators $X$ on $\clh$ such that
\[
T_i^*XT_i=X
\]
for all $i=1,\dots,n$. We show that any positive $T$-Toeplitz operator can be factorized in terms of an isometric pseudo-extension of $T$. A similar factorization result in terms of a BCL type of co-isometric pseudo-extension is also obtained for positive pure lower $T$-Toeplitz operators. However, 
a certain difference has been observed between the case $n=2$ and $n>2$. In a more general context, by considering $n$-tuples of commuting contractions $S$ and $T$, we also study $(S, T)$-Toeplitz operators. 
\end{abstract}

\maketitle

\section*{Notation}
\begin{list}{\quad}{}
 \item $\mathbb{Z}_+$ \quad \quad  Set of all positive integer numbers including 0.
\item $\mathbb{Z}^n_+$ \quad \quad $\{\alpha=(\alpha_1,\ldots, \alpha_n): \alpha_i \in \mathbb{Z}_+, i=1,\ldots,n\}$.
\item $\mathbb{C}$ \quad \quad \, Set of all complex numbers.
\item $T$ \quad \quad \, $n$-tuple of commuting bounded operators $(T_1, \ldots, T_n)$.
\item $T^{\alpha}$ \quad \quad $T_1^{\alpha_1} \cdots T_n^{\alpha_n}$.
\item $\mathbb{D}$  \quad \quad \, Open unit disc $\{z\in\mathbb{C}: |z|<1\}$.
\end{list}

For two Hilbert spaces (separable) $\clh$ and $\clk$, we shall denote the space of all bounded operators from $\clh$ into $\clk$ by $\clb(\clh, \clk)$, and we abbreviate $\clb(\clh, \clh)$ to $\clb(\clh)$. For a closed subspace $\cls$ of $\clh$, we denote by  $P_\cls$ the orthogonal projection from $\clh$ onto $\cls$, and the restriction of an operator $A\in \clb(\clh)$ to $\cls$ will be denoted by $A|_{\cls}$. 
We consider $\Z_+^n$ with the usual partial order, that is, for two multi-indexes $\alpha=(\alpha_1,\ldots,\alpha_n) \in \Z^n_+$ and $\beta=(\beta_1,\ldots,\beta_n) \in \Z^n_+$, we say $\alpha\leq \beta$ if and only if $\alpha_i\leq \beta_i$ for all $i=1,\ldots,n$. 

\newsection{Introduction}

Toeplitz operators on the Hardy space are one of the most important and widely studied objects in operator theory as well as in function theory, and it was first introduced by O. Toeplitz in 1911 (\cite{OT}). Toeplitz operator with a symbol $\psi \in L^\infty (\mathbb{T})$ on the Hardy space $H^2(\D)$ is denoted by $T_\psi$ and defined as
\[T_\psi f= P_{H^2(\D)} \psi f \quad ( f \in H^2(\D)).\]
Here, $L^\infty (\mathbb{T})$ denotes the $C^*$-algebra of $\mathbb{C}$-valued essentially bounded Lebesgue measurable functions on the unit circle $\mathbb{T}$, $L^2 (\mathbb{T})$ is the Hilbert space of all square-integrable functions on $\mathbb{T}$, and the Hardy space $H^2(\D)$ over $\D$ is defined as 
\[H^2(\D)=\{f: \D \to \mathbb{C}: f(z)=\sum_{k\in \Z_+} a_n z^k, a_k\in \mathbb{C}, z \in \D, \sum_{k\in \Z_+} |a_k|^2 < \infty \}.\]
On $H^2(\D)$, one of the distinguished isometries is the unilateral shift $M_z$, defined by $M_zf(w)=wf(w)$ ($f\in H^2(\D)$).   
The Brown and Halmos characterization of Toeplitz operators (\cite{BH}) 
says that an operator $X\in \clb(H^2(\D))$ is a Toeplitz operator if and only if $M_z^* X M_z=X$, that is, $X=T_\psi$ for some $\psi \in L^\infty (\mathbb{T})$ if and only if $M_z^* X M_z=X$. 
In an abstract setting, replacing the unilateral shift $M_z$ by a Hilbert space contraction $T\in \clb(\clh)$, several authors have studied $T$-Toeplitz operators which consists of $X\in\clb(\clh)$ satisfying the identity $T^* X T =X$. We refer the reader to \cite{LK, LK1, CM, CM2} and references therein. In the multi-variable setting, for an $n$-tuple of commuting contractions $T=(T_1,\ldots, T_n)\in\clb(\clh)^n$, an operator $X\in\clb(\clh)$ is a $T$-Toeplitz operator if
\begin{align*}
    T^*_i X T_i=X \quad (i=1,\ldots,n).
\end{align*}
Such operators have been considered in \cite{BTH, BP, PM}. In particular, if we take $T=(M_{z_1},\ldots,M_{z_n})$, the $n$-tuple of multi-shifts on the Hardy space $H^2(\D^n)$ over the polydisc $\D^n$, then $T$-Toeplitz operators are precisely the Toeplitz operators on $H^2(\D^n)$ (\cite{JMS}). It has been shown in \cite{BTH} that the existence of non-zero $T$-Toeplitz operators is related to isometric pseudo-extensions of $T$.  
\begin{Definition}\label{pseudo_D1}
Let $T=(T_1,\ldots,T_n)$ be an $n$-tuple of commuting commuting contractions on $\clh$. An $n$-tuple of commuting commuting contractions $V=(V_1,\ldots,V_n)$ on $\clk$ is said to be a pseudo-extension of $T$ if 
\item (1) there is a non-zero contraction $\clj:\clh \to \clk$, and 
\item (2) $\clj T_i=V_i \clj$ for all $i=1,\ldots,n$.
\end{Definition}
A pseudo-extension of $T$ is denoted by a triple $(\clj,\clk, V)$. A pseudo-extension $(\clj,\clk, V)$ of $T$ is said to be minimal if $\clk$ is the smallest joint reducing space for $V$ containing $\clj \clh$.
A minimal pseudo-extension $(\clj, \clk, V)$ of $T$ is called canonical if
\[\clj^* \clj= \text{SOT}-\lim_{k\to \infty} P^{* k}P^k,\]
where $P=T_1\cdots T_n$ is the product contraction of $T$ and the limit is in the strong operator topology (SOT). A pseudo-extension $(\clj,\clk, V)$ of $T$ is said to be isometric (co-isometric or unitary) pseudo-extension if $V=(V_1,\ldots,V_n)$ is an $n$-tuple of commuting isometries (co-isometries or unitaries, respectively). In the case when $V$ is an $n$-tuple of non-commuting operators, we say that the triple $(\clj,\clk, V)$ is a non-commuting pseudo-extension of $T$. In \cite{BTH}, it is proved that a non-zero $T$-Toeplitz operator exists if and only if $T$ has a canonical isometric pseudo-extension $(\clj,\clk, V)$. In fact, if $(\clj,\clk, V)$ is any isometric pseudo-extension of $T$, then it is easy to see that $\clj^* \clj$ is a non-zero contractive positive $T$-Toeplitz operator. This leads us to ask a natural question whether all positive contractive $T$-Toeplitz operators are of the form $\clj^* \clj$ for some isometric pseudo-extension $(\clj,\clk, V)$ of $T$. We answer this question in affirmative. The answer was known to be positive for the case $n=1$ (see ~\cite{LK}).      
\begin{Theorem}\label{Fact_To}
Let $T=(T_1,\ldots,T_n) \in \clb(\clh)^n$ be an $n$-tuple of commuting contractions. Then $R$ is a positive $T$-Toeplitz operator if and only if there exist a Hilbert space $\clk$, a bounded operator $\clj: \clh \to \clk$, and an $n$-tuple of commuting isometries $V=(V_1,\ldots, V_n)\in \clb(\clk)^n$ such that 
\[\clj^* \clj =R \quad \text{and}\quad \clj T_i= V_i \clj \quad (i=1,\ldots,n).\] 
In particular, if $R$ is a positive contractive $T$-Toeplitz operator, then  $(\clj, \clk, V)$ is an isometric pseudo-extension of $T$. 
\end{Theorem}
Along this line, for an $n$-tuple of commuting contractions $T=(T_1, \ldots,T_n)\in\clb(\clh)^n$, we also study solutions of systems of operator inequalities
\begin{align}\label{In3}
    T^*_i X T_i \geq X \quad (i=1,\ldots,n), 
\end{align}
and
\begin{align}\label{In4}
    T^*_i X T_i \leq X \quad (i=1,\ldots,n), 
\end{align}
where $X\in\clb(\clh)$ is a self-adjoint operator. We say that a self-adjoint operator $X\in\clb(\clh)$ is an upper (lower) $T$-Toeplitz operator if $X$ satisfies \eqref{In3} (\eqref{In4}, respectively). For $n=1$, solutions of such inequalities have been studied in \cite{DG}(see also \cite{GC}). Similar to the case of $n=1$, the upper and lower $T$-Toeplitz operators are related in the following way. For $\alpha\in\mathbb Z_+^n$, by $\alpha\to \infty$ we mean that each $\alpha_i\to \infty$. 
 \begin{Theorem}\label{sb1}
Let $T=(T_1, \ldots,T_n)\in \clb(\clh)^n$  be an $n$-tuple of commuting contractions. Then 
$X$ is an upper (lower) $T$-Toeplitz operator if and only if $X$ can be uniquely written as $X=U-N$ ($X=U+N$), where $U$ is a self-adjoint $T$-Toeplitz operator and $N$ is a positive lower $T$-Toeplitz operator 
such that $T^{* \alpha} N T^{\alpha} \to 0$ in SOT as $\alpha\to \infty$.
\end{Theorem}
Motivated by the above theorem, we define 
\[\clp\cll_+(T)=\{N\in\clb(\clh):N\geq0, T^*_i N T_i\leq N, i=1\ldots,n,T^{* \alpha} N T^{\alpha} \to 0\, \text{in SOT as}\, \alpha\to\infty \}\]
and we call elements of $\clp\cll_+(T)$ as positive pure lower $T$-Toeplitz operators. It can be shown that, in the above definition of $\clp\cll_+(T)$, the convergence $T^{* \alpha} N T^{\alpha} \to 0$ in SOT as $\alpha\to \infty$ is equivalent to $P_T^{* k}N P_T^{k}\to 0$ in SOT as $k\to \infty$, where $P_T=T_1 \cdots T_n$.
From the point of view of the above theorem, $\clp\cll_+(T)$ plays a major role in the study of upper or lower $T$-Toeplitz operators and it is important to find a characterization of $\clp\cll_+(T)$.
In many contexts in operator theory, the difference between the case $n=2$ and $n>2$ is prominent. For instance, any pair of commuting contractions always possesses an isometric dilation (\cite{Ando}) but any $n$-tuples ($n>2$) of commuting contractions does not possess an isometric dilation in general (\cite{Par}). Such a distinction also arises in our characterization of positive pure lower $T$-Toeplitz operators and which is why we separate the case $n=2$ from the rest. Before further details, recall that due to Berger, Coburn and Lebow (\cite{BCL}), any $n$-tuple of commuting isometies with product isometries being pure is unitarily equivalent to $n$-tuple of commuting isometries $(M_{\Phi_1},\dots, M_{\Phi_n})$ on the Hardy space $H^2_{\cle}(\D)$, where $\cle$ is a coefficient Hilbert space and for all $i=1,\dots,n$,
\[
\Phi_i(z)=U_i (P_i^\perp+ z P_i), \,\, ( z\in \D)
\]
for some unitary $U_i\in \clb(\clh)$ and orthogonal projection $P_i\in \clb(\cle)$ such that 
\begin{equation*}
    \Phi_i(z)\Phi_j(z)=\Phi_j(z)\Phi_i(z)\ \text{ and } \Phi_1(z)\cdots\Phi_n(z)=zI_{\cle}\ \quad (z\in\mathbb D).
\end{equation*}
Here, the multiplication operator $M_\Phi$ corresponding to $\Phi$ is defined as
 \[M_\Phi:H^2_{\cle}(\mathbb D) \to H^2_{\cle}(\D), \quad f\mapsto \Phi f \quad (f\in H^2_{\cle}(\mathbb D)).\]
Such an $n$-tuple $(M_{\Phi_1},\cdots, M_{\Phi_n})$ is commonly known as a BCL tuple (see ~\cite{BCL, DDS, BDF} for more details). 
For $n=2$, if $T=(T_1, T_2)$ is a pair of commuting contractions, then we prove that any contractive positive pure lower $T$-Toeplitz operators can be factorized in terms of a BCL pair of co-isometric pseudo-extension of $T$. This is one of the main results of this paper and it is motivated from its one variable counterpart (see ~\cite[Theorem 5]{DG}). A more detailed statement of the result is the following. 
\begin{Theorem}\label{2_M_theta}
Let $T=(T_1,T_2)\in \clb(\clh)^2$  be a pair of commuting contractions. Then $N \in \clp\cll_+(T)$ if and only if there exist a Hilbert space $\cle$, a bounded operator $\Pi:\clh \to H^2_{ \cle}(\mathbb{D})$, and a pair of commuting isometries $M=(M_{\Phi_1}, M_{\Phi_2})\in \clb(H^2_{ \cle}(\mathbb{D}))^2$  such that 
\[N=\Pi^* \Pi,
\quad \text{and}\quad \Pi T_i= M^*_{\Phi_i}\Pi \quad(i=1,2),\]
where
$\Phi_1(z)=(P+ z P^\perp) U^*$ and $\Phi_2(z)=U(P^\perp+ z P)$ are so that $\Phi_1(z) \Phi_2(z)=\Phi_2(z)\Phi_1(z)=z$ ($z \in \D$) for some unitary $U\in\clb(\cle)$ and orthogonal projection $P\in\clb(\cle)$.

In particular, if $N\in \clp\cll_+(T)$ is a contraction, then $(\Pi, H^2_{ \cle}(\mathbb{D}), (M^*_{\Phi_1}, M^*_{\Phi_2}))$ is a co-isometric pseudo-extension of $T$.
\end{Theorem}

 We prove this theorem in Section \ref{se3} as Theorem~\ref{Re: 2_M_theta}. The above theorem is related to isometric dilations of pairs of commuting contractions. Recall that a commuting $n$-tuple of isometries $V=(V_1,\dots,V_n)$ on $\clk$ is an isometric dilation of an $n$-tuple of commuting contractions $T=(T_1,\dots,T_n)$ on $\clh$ if there is an isometry $\Pi:\clh\to \clk$ such that  
 \[
\Pi T_i^{*}= V_i^*\Pi
 \]
 for all $i=1,\dots,n$.  Here we say that an operator $A\in \clb(\clh)$ is pure if $A^{* k}\to 0$ in SOT as $k\to \infty$. Let $T=(T_1,T_2)$ be a pair of commuting contractions such that the product contraction $P_T=T_1T_2$ is pure. Then it can be easily verified that $I\in \clp\cll_+(T^*)$. Therefore, by applying the above theorem for $N=I$, we see that $(T_1,T_2)$ dilates to a BCL pair of isometries $(M_{\Phi_1}, M_{\Phi_2})$. Such a dilation result is known (see ~\cite{BSS}). In fact, the explicit construction of $\Phi_1$ and $\Phi_2$ obtained in the above theorem uses techniques found in ~\cite{BSS}.

 For an $n$ ($n>2$)-tuples of commuting contractions $T$, we lose commutativity in the characterization of $\clp\cll_+(T)$. We prove that any contractive positive pure lower $T$-Toeplitz operator can be factorized in terms of a BCL type of non-commuting co-isometric pseudo-extension of $T$. 
\begin{Theorem}\label{M_theta}
Let $T=(T_1, \ldots,T_n)\in \clb(\clh)^n$ be an $n$-tuple of commuting contractions with $n>2$. Then $N\in \clp \cll_+(T)$ if and only if there exist a Hilbert space $\cle$, a bounded operator $\Pi:\clh \to H^2_{ \cle}(\mathbb{D})$, and an $n$-tuple of isometries $M=(M_{\Theta_1}, \ldots, M_{\Theta_n})$ on $H^2_{ \cle}(\mathbb{D})$  such that 
\[N=\Pi^* \Pi\, \text{and } \Pi T_i= M^*_{\Theta_i}\Pi  
  \quad (i=1,\ldots,n),\]
where for each $i=1,\ldots,n$, 
$\Theta_i(z)=(P_i+zP_i^{\perp})U_i^*$ for some unitary $U_i\in \clb(\cle)$ and orthogonal projection $P_i\in\clb(\cle)$.
Moreover, 
\[
M^*_{z} \Pi=M^*_{\Theta_1}M^*_{\Theta_2} \dots M^*_{\Theta_n}\Pi.
\]

In particular, if $N\in \clp \cll_+(T)$ is a contraction, then $(\Pi, H^2_{ \cle}(\mathbb{D}), (M^*_{\Theta_1},\ldots, M^*_{\Theta_n}))$ is a non-commuting co-isometric pseudo-extension of $T$.
\end{Theorem}

This theorem is proved in Section \ref{se3} as Theorem~\ref{Re: M_theta}. The main difference with Theorem~\ref{2_M_theta} is that the $n$-tuple of BCL type of isometries $(M_{\Theta_1},\dots,M_{\Theta_n})$ on $H^2_{\cle}(\D)$ are not commuting, in general. As before, if we assume that the product contraction $P_T=T_1\cdots T_n$ is pure, then $I\in \clp\cll_+(T^*)$ and the $n$-tuple of isometries $(M_{\Theta_1},\dots,M_{\Theta_n})$ on $H^2_{\cle}(\D)$ gives a Halmos dilation of $T$, that is, for each $i=1,\dots,n$, $M_{\Theta_i}$ is an isometric dilation of $T_i$.

To generalize the abstract notion of Toeplitz operators in a more general setting, R. G. Douglas (\cite{DG}) considered the following operator identity 
\[S^* X T=X \ (X\in \clb(\clh, \clk)),\]
for two contractions $S\in \clb(\clk)$ and $T\in\clb(\clh)$. A lot of research is devoted to studying solutions of such an operator identity, see \cite{GC, DG, Ptv} and references therein. In the multi-variable setting, considering two $n$-tuples of commuting contractions $S=(S_1,\ldots,S_n)\in \clb(\clk)^n$ and $T=(T_1, \ldots, T_n)\in \clb(\clh)^n$, we have also studied the following system of operator equations
\begin{align*}
    S^*_i X T_i =X \quad  \quad (i=1,\ldots,n),
\end{align*}
where $X\in \clb(\clh, \clk)$. The space of all solutions of the above operator equations is denoted by $\mathcal T(S,T)$ and we call elements in $\mathcal T(S,T)$  as  $(S, T)$-Toeplitz operators. 
In Section \ref{Sec4}, a general construction of $(S,T)$-Toeplitz operators in terms of Banach limits is given (see Theorem \ref{B_limt}). In this context, it is a natural question to ask when the space $\mathcal T(S,T)$ is non-trivial? Satisfactory answer to this question is not known. However, we find some necessary conditions in terms of canonical isometric pseudo-extension of $S$ and $T$ for the space $\mathcal T(S,T)$ to be non-trivial. In particular, we show that if $\mathcal T(S,T)$ is non-trivial then $S$ and $T$ have canonical isometric pseudo-extensions. One could then ask the following immediate question: If $\mathcal T(S,T)$ is non-trivial and if $(\clj_T, \clq_T,V)$ and $(\clj_S,\clq_S,W)$ are canonical isometric pseudo-extensions of $T$ and $S$, respectively, then how $\mathcal T(S,T)$ and $\mathcal T(W,V)$ are related? The answer to this question is rather neat as we show that 
$$\T(S,T)=\clj_S^* \T(W, V) \clj_T.$$ 
This is proved in Theorem~\ref{Th2}.

The paper has four sections and the plan of the paper is as follows. In Section 2, we study $T$-Toeplitz operators. The study of upper and lower $T$-Toeplitz operators is contained in Section 3. In the last section we consider $(S, T)$-Toeplitz operators.

  \newsection{$T$-Toeplitz operators}
Let $T=(T_1,\dots,T_n)$ be an $n$-tuple of commuting contractions on $\clh$. Recall that an operator $X\in \clb(\clh)$ is a $T$-Toeplitz operator if  
\[
T_i^*XT_i=X
\]
for all $i=1,\dots,n$.
The aim of this section is to study various properties of $T$-Toeplitz operators. We begin with the equivalent criteria for the space of all $T$-Toeplitz operators $\mathcal T(T)$ to be non-trivial. This result is proved in Theorem 1.3 of \cite{BTH} and can also be proved using results in ~\cite{PM}. A similar kind of result is also proved in an another context in \cite{BDS}. We include a proof here as we shall use the construction in the proof of $(2)\implies (3)$ in the latter section. 


\begin{Theorem}[cf. ~\cite{BTH}]\label{non-zero Toeplitz}
 Let $T = (T_1, \ldots, T_n)$ be an $n$-tuple of commuting contractions on $\clh$. Then the following are equivalent.
\begin{itemize}
    \item[(1)] $\T(T)$ is non-zero.
\item[(2)] The adjoint of $P$ is not pure, where $P=T_1\cdots T_n$.
\item[(3)] There exists a canonical isometric pseudo-extension of $T$.
\end{itemize}
\end{Theorem}

\begin{proof}[\text{Proof of} $(1)\implies (2)$:]
We prove by the contradictory method. Let us assume that the adjoint of $P$ is pure. If $X$ is a non-zero element in $\T(T)$,
then 
\begin{align} \label{pure_1}
    P^{* k}X P^k=X \quad (k\in \Z_+).
\end{align}
Now, for $h \in \clh$, 
\begin{align*}
    \|X h\|&=\|P^{*k} X P^k h\|
     \leq \|P^{* k}\| \|X\| \|P^k h\|
    \leq \|X\| \|P^k h\| 
     \to 0 \quad \text{as } k\to \infty.  
\end{align*}
Therefore, $X=0$, which contradicts that $X$ is a non zero operator.

\noindent\text{\textit{Proof of}} $(2)\implies (3)$:
Let us assume that  $P^k \nrightarrow 0$ in SOT. Then the positive operator $Q_T$, defined by
 \begin{align}\label{Q_T}
     Q^2_T=\text{SOT}-\lim_{k\to \infty} P^{* k}P^k,
     \end{align}
 is non-zero and satisfies 
\[P^* Q_T^2P=Q_T^2.\]
 By the above identity, we have an isometry $\mathsf{V}: \clq_T \to \clq_T$, defined by \[ Q_T h \mapsto Q_T P h\quad (h\in\clh),\] 
 where $ \clq_T=\overline{\text{Ran}}\, Q_T$.
Since for each $i=1,\ldots,n$, $T_i$ is a contraction, we have 
 \[\langle T^*_i Q_T^2 T_i h, h\rangle= \lim_{k\to \infty} \langle P^{*k}(T^*_i T_i)P^k h, h\rangle \leq \lim_{k\to \infty} \langle P^{*k}P^k h,h\rangle =\langle Q_T^2 h,h\rangle.\]
Therefore,
\[T^*_i Q_T^2 T_i \leq Q_T^2\quad (i=1,\ldots,n).\]
By virtue of the Douglas' factorization lemma (\cite{D}), we obtain a contraction $V_i: \clq_T \to \clq_T$ satisfying the following relation for $h\in \clh$,
\[V_i Q_T h=Q_T T_i h \quad (i=1,\ldots, n).\]
Using the commutativity of the tuple $T$, we have for $i,j=1, \ldots,n$ and $h \in \clh$,
\[V_i V_j Q_T h=V_i Q_T T_j h=Q_T T_i T_j h=V_j Q_T T_i h=V_j V_i Q_T h. \]
Therefore, $(V_1,\ldots,V_n)$ is an $n$-tuple of commuting contractions. Also, since $P=T_1 \cdots T_n$, a similar computation as above yields
\[\mathsf{V}=V_1 \cdots V_n.\]
The above identity together with the fact that $\mathsf{V}$ is an isometry implies each $V_i$ is an isometry. 
Finally, consider the map $\clj_T : \clh \to \clq_T $, defined by $\clj_T h = Q_T h$  for $h\in \clh$. It is now clear that $\clj_T$ is a contraction,
 \begin{align}\label{Cano_qus}
     \clj_T^* \clj_T= \text{SOT}-\lim_{k \to \infty} P^{* k}P^k \quad \text{and}\quad V_i \clj_T=\clj_T T_i \quad (i=1,\ldots,n).
 \end{align}
 Thus, $(\clj_T, \clq_T, (V_1,\dots,V_n))$ is a canonical isometric pseudo-extension of $T$.
 
 \noindent\text{\textit{Proof of}} $(3)\implies (1)$: Let $(\clj, \clk, W=(W_1, \ldots, W_n))$ be an isometric pseudo-extension of $T$. Now, for $i=1,\ldots,n$,
 \[T^*_i \clj^* \clj T_i=\clj^* W^*_i W_i \clj=\clj^* \clj.\]
Hence, $\clj^*\clj$ is a non-zero $T$-Toeplitz operator.
\end{proof}

From the above theorem, one observes the following.
\begin{Remarks}\label{Remark_positive} $\textup{(i)}$ If $R$ is a contractive positive $T$-Toeplitz operator, then for $k\in \Z_+$, 
\[R=P^{* k} R P^k \leq P^{* k} P^k.\]
Therefore, $R \leq \clj_T^* \clj_T$. And, if $R$ is a $T$-Toeplitz operator such that $R \geq I$, then $\clj_T^* \clj_T \leq R$.

$\textup{(ii)}$
The existence of a non-zero $T$-Toeplitz operator is equivalent to the existence of a positive $T$-Toeplitz operator of the form $\clj^*_T \clj_T$, where $(\clj_T, \clq_T, (V_1,\dots,V_n))$ is a canonical isometric pseudo-extension of $T$.
     
\end{Remarks}

We now proceed to prove Theorem~\ref{Fact_To} which says that any contractive positive $T$-Toeplitz operator is of the form $\clj^*\clj$ for some isometric pseudo-extension $(\clj, \clk, V)$ of $T$. 

\begin{proof}[Proof of Theorem \ref{Fact_To}]
Since $R$ is a positive $T$-Toeplitz operator. Therefore, for $i=1, \ldots ,n$, $T^*_i R^{1/2} R^{1/2} T_i=R^{1/2} R^{1/2}$, and thus,
\[\|R^{1/2}T_i h\|=\|R^{1/2} h\| \quad ( h \in \clh).\] We set $\mathcal R:=\overline{\text{Ran}}\, R^{1/2}$ and consider for each $i=1,\ldots,n$, the isometry $V_i:\mathcal R \to \mathcal R$, defined by $R^{1/2}h \mapsto R^{1/2} T_i h$ $(h \in \clh)$. 
The commutativity of the tuple $T$ reveals that the tuple $V:=(V_1,\ldots, V_n)\in \clb(\mathcal R)^n$ is a commuting $n$-tuple of isometries.
Then the bounded operator $\clj: \clh \to \mathcal R$, defined by $h \mapsto R^{1/2}h$ ($h \in \clh$), satisfies  
\[\clj^*\clj=R\ \text{and }\clj T_i h=R^{1/2} T_i h=V_i R^{1/2}h=V_i \clj h\quad (h \in \clh),\]
for all $i=1,\dots,n$.
 Conversely, if there is a bounded operator $\clj: \clh \to \clk$ and an $n$-tuple of isometries $V=(V_1,\dots,V_n)$ on $\clk$ satisfying $\clj T_i= V_i \clj$ ($i=1,\ldots,n$) and $\clj^* \clj =R$, then for $i=1,\ldots,n$, 
 \[T^*_i R T_i=T^*_i \clj^* \clj T_i=\clj^* V_i^* V_i \clj=\clj^* \clj=R.\]
 Hence, $R=\clj^* \clj$ is a $T$-Toeplitz operator. Moreover, if $R$ is a contraction, then $\clj$ will be a contraction. This completes the proof.
 \end{proof}

We illustrate the above theorem by considering a simple example.
\begin{Example}
   For $n=2$ and $\clh=\mathbb C^l$ ($l\geq 3$). We consider a pair of commuting operators $T=(T_1, T_2)$ on $\mathbb C^l$, defined by
   \[ T_1(e_i)=\begin{cases} 
      e_i & i=1,\ldots, k \\
      0 & i=k+1,\ldots, l
   \end{cases} \quad \text{and} \quad T_2(e_i)=\begin{cases} 
      0 & i=1,\ldots, m-1 \\
      e_i & i=m,\ldots, l
   \end{cases}
\]
where $1<m\leq k <l$ and $\{e_i\}_{i=1}^l$ is the standard orthonormal basis for $\mathbb C^l$, that is, for $i=1,\ldots, l$, $e_i=(0,\ldots,0,\underbrace{1}_{i\text{th place}},0,\ldots,0)$.
Then it is easy to check that for $a=(a_m, \ldots, a_k)$ and $a_i\geq 0$ ($i=m, \ldots, k$), the operator $X_a$, defined by
\[ X_a(e_i)=\begin{cases} 
      0 & i=1,\ldots, m-1 \\
      a_i e_i & i=m,\ldots, k\\
      0 & i=k+1, \ldots, l
   \end{cases}
\]
 is a positive $T$-Toeplitz operator.
 Then the map $\clj:\mathbb C^l \to \mathbb C^{k-m+1}$, defined by
 \[\clj \begin{bmatrix}
     z_1\\
     \vdots\\
     z_l
 \end{bmatrix}=\begin{bmatrix}
     a^{1/2}_m z_m\\
     \vdots\\
     a^{1/2}_k z_k
 \end{bmatrix},\]
 and for each $i=1,\ldots, l$, the isometry $V_i=I_{\mathbb C^{k-m+1}}$ on $\mathbb C^{k-m+1}$ satisfy
 \[\clj^* \clj=X_a \, \text{ and }\, \clj T_i=V_i \clj \,\, (i=1,2).\]
 
\end{Example}

In the rest of the section, we study compact $T$-Toeplitz operators. 
Before proceeding further, for a contraction $S\in \clb(\clh)$, let us consider a subspace of $\clh$, denoted by $\clu_S$ and defined as follows 
$$\clu_S=\overline{\text{Span}}\{h\in \clh: Sh=\alpha h : |\alpha|=1\}.$$


\begin{Proposition}\label{Pro_Compt}
Let $T=(T_1, \ldots,T_n)\in \clb(\clh)^n$  be an $n$-tuple of commuting contractions with $P=T_1\cdots T_n$.
If a compact operator $X$ belongs to $\T(T)$, then 
\item 
(a) $X$ and $X^*$ commute with each $T_i$ and $P$, and
\item
(b) Each $\clu_{T_i}$ and $\clu _P$ is reducing subspace for $X$ and $X|_{\clu_{T_i}^{\perp}}=0$, $X|_{\clu_{P}^{\perp}}=0$.
\end{Proposition}
\begin{proof} 
Let $X\in\T(T)$ be a compact operator. Then $X\in \T(P)$ and also $X\in \T(T_i)$ for all $i=1,\dots,n$. The proof now follows by applying \cite[Theorem 8]{DG} for the compact operator $X$ viewed as $P$-Toeplitz operator as well as $T_i$-Toeplitz operator for all $i=1,\dots,n$.  
\end{proof}


The following corollary is an easy consequence of the above proposition. 
\begin{Corollary}\label{coro1}
Let $T=(T_1, \ldots,T_n)\in \clb(\clh)^n$  be an $n$-tuple of commuting contractions with $P=T_1\cdots T_n$, and let $X\in\T(T)$ be a compact operator. If one of $\clu_{T_i}=\{0\}$ ($i=1,\ldots,n$) or $\clu _P=\{0\}$, then $X=0$.
\end{Corollary}

We end the section with the couple of immediate consequences of the above proposition and the corollary. These might be known to the experts in the area. However, due to lack of references, we include them as corollaries. Recall the Brown-Halmos type characterization of Toeplitz operators on the Hardy space over the polydisc $H^2(\D^n)$ which says that an operator $X\in\clb(H^2(\D^n))$ is a Toeplitz operator if 
\[
M^*_{z_i} X M_{z_i} =X, \quad (i=1,\ldots,n)
\]
where $(M_{z_1}, \ldots, M_{z_n})$ is the $d$-tuple of multi-shifts on $H^2(\D^n)$.
\begin{Corollary}
 A compact Toeplitz operator on the Hardy space over the polydisc is zero. 
\end{Corollary}

\begin{Corollary}
Let $T=(T_1, \ldots,T_n)\in \clb(\clh)^n$  be an $n$-tuple of commuting contractions with $P=T_1\cdots T_n$. Let the positive operator $Q_T$ be as in \eqref{Q_T}. If $Q_T$ is a compact operator, then $Q_T$ is a finite dimensional orthogonal projection and $T_i|_{\clq_T}$ ($i=1\ldots,n$) and $P|_{\clq_T}$ are unitary, where $\clq_T=\text{Ran}\, Q_T$.
\end{Corollary}
\begin{proof}
From the definition of $Q_T$, we have $P^* Q_T^2 P=Q_T^2$. Also, it follows from the inequalities 
\[
P^{* (k+1)}P^{k+1}\le P^{* k} T_i^* T_iP^k\le P^{* k}P^k \quad (1\le i\le n)
\]
that $T^*_i Q_T^2 T_i=Q_T^2$ for all $i=1, \ldots, n$. Then by part (a) of Proposition \ref{Pro_Compt}, we get $T_i Q_T^2=Q_T^2T_i$ ($i=1, \ldots, n$) and $P Q_T^2=Q_T^2 P$. Therefore, using functional calculus, we have $T^*_i T_i Q_T=Q_T$ and $P^* P Q_T=Q_T$. This in particular shows that $P^*Q_TP=Q_T$ and for all $h\in\clh$,
\[
\lim_{k\to \infty} \langle P^{* k} P^k Q_T h,  h \rangle =\langle P^{* k}Q_TP^k h, h\rangle =
\langle Q_T h, h \rangle.
\]Thus, $Q_T^2=Q_T$, and hence, the compactness of $Q_T$ completes the proof.
\end{proof}

\newsection{Lower and upper $T$-Toeplitz operators}\label{se3}

For an $n$-tuple of commuting contractions $T=(T_1, \ldots,T_n)\in \clb(\clh)^n$, we deal with the following systems of operator inequalities: 
\begin{align*}
    T^*_i X T_i \geq X \text{ and }
    T^*_i X T_i \leq X \quad (i=1,\ldots,n), 
\end{align*}
where $X\in\clb(\clh)$ is a self-adjoint operator.
Recall that a self-adjoint operator $X\in \clb(\clh)$ is an upper $T$-Toeplitz operator if $X$ satisfies the system of inequalities of first kind as above and is a lower $T$-Toeplitz operator if it satisfies the system of inequalities of second kind. 
First, we consider the cone of positive upper $T$-Toeplitz operators and we set 
$$\clu_+(T)=\{Q\in \clb(\clh): Q\geq 0, T_i^* Q T_i \geq Q, i=1, \ldots, n\}.$$
 We find a necessary and sufficient condition for the existence of a non-trivial element in  $\clu_+(T)$ next.
\begin{Proposition}
Let $T=(T_1, \ldots,T_n)\in \clb(\clh)^n$  be an $n$-tuple of commuting contractions with $P=T_1\cdots T_n$. Then $\clu_+(T) = \{0\}$ if and only if the adjoint of $P$ is pure.
\end{Proposition}

\begin{proof}
Let the adjoint of $P$ is pure, and $X \in \clu_+(T)$. Then, for $i=1,\ldots,n$, $X \leq T^{*}_i X T_i$ and this implies that $X\leq P^{* k}X P^k$ for $k\in \Z_+$.
Now, for $h\in \clh$, 
\begin{align*}
     \|P^{* k} X P^k h\| &\leq \|P^{* k}\| \|X\| \|P^k h\|
    \leq \|X\| \|P^k h\| \to 0\quad \text{as }k\to \infty.
\end{align*}
Thus, $X \leq 0$, and therefore, $X=0$.
Hence, $\clu_+ (T) = \{0\}$. Conversely, let $\clu_+ (T) = \{0\}$. Suppose that the adjoint of $P$ is not pure. Then by Theorem~\ref{non-zero Toeplitz}, $\T(T)$ is non-trivial. In fact, by part(ii) of Remark~\ref{Remark_positive}, we have a positive non-zero element in $\T(T)$, namely $\clj_T^*\clj_T$ for some canonical pseudo-extension $(\clj_T,\clk,V)$ of $T$. In particular, $\clj_T^*\clj_T\in \clu_+(T)$ which contradicts the fact that $\clu_+ (T) = \{0\}$. This completes the proof.
\end{proof}
Unlike upper $T$-Toeplitz operators, the cone of positive lower $T$-Toeplitz operators is always non-trivial as the identity operator is a lower $T$-Toeplitz operator. We now prove Theorem~\ref{sb1} which establishes the relation between upper and lower $T$-Toeplitz operators.

\begin{proof}[Proof of Theorem \ref{sb1}]
We only prove the theorem for upper $T$-Toeplitz operator as the argument for the case of lower $T$-Toeplitz operator is similar.  Let us assume that $X$ is an upper $T$-Toeplitz operator. Then, for any  $\alpha=(\alpha_1,\ldots,\alpha_n) \in \mathbb{Z}^n_+$, $T^{ *\alpha} X T^\alpha \geq X$. 
Set $U_\alpha:=T^{ *\alpha} X T^\alpha$, for $\alpha\in \Z_+^n$. For $\beta \in \Z_+^n$ and $\beta \geq \alpha$, 
\[U_\beta-U_\alpha=T^{ *\beta} X T^\beta-T^{ *\alpha} X T^\alpha=T^{* \alpha}(T^{* (\beta-\alpha)} X T^{(\beta- \alpha)} - X)T^\alpha \geq 0 \]
Thus $\{U_\alpha\}_{\alpha \in \Z_+^n}$ is an increasing net of self-adjoint operators. Since 
$\|U_\alpha\|\leq \|X\|$ for all $\alpha \in \Z_+^n$, $\{U_\alpha\}_{\alpha \geq 0}$ converges in the strong operator topology to $U$, say. Clearly, $X \leq U_{\alpha}\leq U$ for all $\alpha \in \Z_+^n$. 
We set $N:= U-X$. Then $N$ is a positive operator.
Now, for $e_i=(0,\ldots,0,\underbrace{1}_{i\text{th place}},0,\ldots,0) \in \Z_+^n$, $T_i^* U_\alpha T_i= T^{* (\alpha + e_i)} X T^{(\alpha + e_i)}$.
Therefore, for all $i=1,\dots,n$, 
\[ T^*_i U T_i= U,\,\, \text{and}\,\, T_i^* N T_i = T_i^* U T_i -T_i^* X T_i
      =U- T_i^* X T_i
      \leq U -X=N.\]
Also, for $\alpha \in \Z_+^n$,
\begin{align*}
    T^{* \alpha} N T^\alpha &=T^{* \alpha} U T^\alpha -T^{* \alpha} X T^\alpha 
    = U -T^{* \alpha} X T^\alpha \to 0 \ \text{in SOT as } \alpha \to \infty.
\end{align*}
 Hence, we have a self-adjoint operator $U$ and a positive operator $N$ such that $X=U-N$ and for $i=1,\ldots,n$, $T_i^* U T_i=U$ and $T_i^*N T_i \leq N$  with $T^{* \alpha} N T^{\alpha} \to 0$ in SOT.
 Converse implication is easy to check.
 
 For the uniqueness of such a decomposition, we assume that $X$ has an another decomposition, that is, 
   $ X= U_1 - N_1$, such that for $i=1, \ldots,n$, $T^*_i U_1 T_i=U_1$ and $T^*_i N_1 T_i \leq N_1$ with $T^{* \alpha} N_1 T^\alpha \to 0$ in SOT as $\alpha\to \infty$. Since $U-U_1=N-N_1$, we have 
   \begin{align*}
       U-U_1=T^{* \alpha} (N-N_1) T^\alpha \to 0 \quad \text{in SOT as } \alpha\to \infty.
   \end{align*}
Thus, $U=U_1$, and therefore $N=N_1$ as asserted.
\end{proof}

By the above theorem it is important to study positive lower $T$-Toeplitz operator $N$ such that 
$T^{* \alpha} N T^{\alpha} \to 0$ in SOT. For such an operator $N$, $T^{* \alpha} N T^{\alpha} \to 0$ in SOT is equivalent to $P_T^{*k} N P_T^k \to 0$ in SOT as $k\to \infty$, where $P_T=T_1\cdots T_n$. Indeed, if $N$ is a positive lower $T$-Toeplitz operator then the equivalence  can be obtain from the following inequalities
\[
P_T^{* k}NP_T^k\le T^{* \alpha}N T^{\alpha}\le P_T^{* m}N P_T^{m},
\]
where for $\alpha\in \Z_+^n$, $m= \text{min}\{ \alpha_i: i=1,\dots,n\}$ and $k=\text{max}\{\alpha_i:i=1,\dots,n\}$.
The above observation leads us to make the following equivalent definition of the cone of pure positive lower $T$-Toeplitz operators
\[
\clp\cll_+ (T)=\{ N\in\clb(\clh): N\ge 0, T_i^*NT_i\le N, i=1,\dots,n, P_T^{*k}N P_T^k\to 0\ \text{in SOT as } k\to \infty\}.
\]
The rest of this section is devoted to characterize the set $\clp\cll_+(T)$, which is our main result in this section. To this end, we need to first introduce some terminologies.

For two Hilbert spaces $\cle_1$ and $\cle_2$, an operator-valued analytic function $\Theta: \D \to \clb(\cle_1, \cle_2)$ is said to be multiplier from 
$H^2_{\cle_1}(\D)$ to $H^2_{\cle_2}(\D)$ if $\Theta f \in H^2_{\cle_2}(\D)$ for $f \in H^2_{\cle_1}(\D)$.
Here, for a Hilbert space $\cle$,
\[H^2_{\cle}(\D)=\{f: \D \to \cle: f(z)=\sum_{k\in \Z_+} a_k z^k, a_k\in \cle, z \in \D, \sum_{k\in \Z_+} \|a_k\|^2 < \infty \}\]
is the $\cle$-valued Hardy space over $\D$. The multiplier $\Theta$ is said to be inner  multiplier if the associated multiplication operator $M_\Theta :H^2_{\cle_1}(\D) \to H^2_{\cle_2}(\D)$, defined by $M_\Theta (f)=\Theta f$ for $f\in H^2_{\cle_1}(\D)$, is a isometry. Let
\[W=\begin{bmatrix}
A & B\\
C & D
\end{bmatrix}:\cle_1 \oplus \cle_2\to \cle_1 \oplus \cle_2 \]
be a unitary. Then the $\clb(\cle_1)$-valued analytic function $\tau_W$ on $\D$, defined by 
\[\tau_W(z):=A+ zB(I-zD)^{-1} C \quad (z\in \D),\]
is called the transfer function corresponding to $W$. However, in this article, we will only deal with transfer functions corresponding to unitary operators of the form $W=\begin{bmatrix}
A & B\\
C & 0
\end{bmatrix}$, and in such cases $\tau_W$ is always a $\clb(\cle_1)$-valued inner function (\cite{NF}). In other words, $\tau_W$ is an inner multiplier on $H^2_{\cle_1}(\D)$. 

We now prove a crucial lemma which does the heavy lifting of our main result. Similar type of result is also proved in a slightly different setup in ~\cite{BSS, MB} and has its root in ~\cite{DS}. Since we prove the lemma in a more general setting, it may appear to be abstract to the reader. Roughly speaking, the lemma finds a sufficient condition on a unitary so that the transfer function of its adjoint gives a co-extension of certain operator.      
\begin{Lemma}\label{sb3}
Let $\sq,\sr, \s,\st\in \clb(\clh)$ be such that $\clr=\overline{\text{Ran}}\, \sr$, $\cls= \overline{\text{Ran}}\, \s$ and $\sq \st=\st \sq$, and let $\sv: \clr\to \cle$ be an isometry for some Hilbert space $\cle$. Suppose that $\pi: \clh \to H^2_{\clr}(\D)$, defined by $h\mapsto \sum_{k\ge 0}z^k \sr \sq^k h$, is a bounded operator and
\[\sw : \begin{bmatrix}
    \sa & \sanb \\
    \sanc & 0
  \end{bmatrix} : \cle \oplus(\tilde{\cle}\oplus \cls) \to  \cle \oplus (\tilde{\cle}\oplus\cls) \] is a unitary for some Hilbert space $\tilde{\cle}$ such that
  \[\sw(\sv \sr h,0_{\tilde{\cle}}, \s \sq h)= (\sv \sr \st h, 0_{\tilde{\cle}}, \s h) \quad (h \in \clh).\] 
Then the transfer function corresponding to the unitary $\sw^*$, $\Phi(z)= \sa^* + z \sanc^* \sanb^* \quad (z \in \D)$,
is an $\clb(\cle)$-valued inner function which satisfies 
  \[\Pi \st = M^* _{\Phi} \Pi,\]
  where $\Pi=(I_{H^2(\D)}\otimes \sv)\pi$. 
\end{Lemma}

\begin{proof}
Since 
\[\begin{bmatrix}
    \sa & \sanb\\
    \sanc & 0
  \end{bmatrix} \begin{bmatrix} \sv \sr h \\ (0_{\tilde{\cle}}, \s \sq h ) \end{bmatrix} =\begin{bmatrix} \sv \sr \st h \\ (0_{\tilde{\cle}}, \s h) \end{bmatrix},\]
 we have 
  \[\sv \sr \st h = \sa \sv \sr h + \sanb (0_{\tilde{\cle}},\s \sq h) \quad \text{and} \quad (0_{\tilde{\cle}},  \s h) = \sanc \sv \sr h.\] 
 Consequently, \[\sv \sr \st = \sa \sv \sr + \sanb \sanc \sv \sr \sq.\]
Now, for $h\in \clh$, $\eta \in \cle$,
\begin{align*}
\langle { M^*_{\Phi} \Pi h, z^l \eta} \rangle =& \langle {\Pi h,  M_{\Phi} (z^l \eta)} \rangle\\
=& \langle {(I_{H^2(\mathbb{D})} \otimes \sv) \sum_{l\in \Z_+}(\sr \sq^l h) z^l, (\sa^* + z \sanc^* \sanb^*)(z^l \eta)}\rangle\\
=& \langle {\sv \sr \sq^l h, \sa^* \eta} \rangle + \langle{\sv \sr \sq^{l+1}h, \sanc^* \sanb^* \eta}\rangle\\
=&  \langle{(\sa \sv \sr \sq^l + \sanb \sanc \sv  \sq^{l+1})h, \eta} \rangle\\
=& \langle{\sv \sr \st \sq^l h, \eta} \rangle,
\end{align*}
 and  
  \begin{align*}
      \langle{\Pi \st h, z^l \eta} \rangle =& \langle{(I_{H^2(\mathbb{D})} \otimes \sv) \sum_{l\in \Z_+}(\sr \sq^l \st h)z^l, z^l \eta} \rangle\\
      =& \langle{\sv \sr \sq^l \st h , \eta} \rangle\\
      =& \langle{\sv \sr \st \sq^l h, \eta}\rangle.
     \end{align*}
     Hence, $\Pi \st = M^* _{\Phi} \Pi$ as asserted.
\end{proof}

We are now ready to prove Theorem ~\ref{2_M_theta}. For the convenience of the reader, we again state the theorem here.  
\begin{Theorem}\label{Re: 2_M_theta}
Let $T=(T_1,T_2)\in \clb(\clh)^2$  be a pair of commuting contractions. Then $N \in \clp\cll_+(T)$ if and only if there exist a Hilbert space $\cle$, a bounded operator $\Pi:\clh \to H^2_{ \cle}(\mathbb{D})$, and a pair of commuting isometries $M=(M_{\Phi_1}, M_{\Phi_2})\in \clb(H^2_{ \cle}(\mathbb{D}))^2$  such that 
\[N=\Pi^* \Pi,
\quad \text{and}\quad \Pi T_i= M^*_{\Phi_i}\Pi \quad(i=1,2),\]
where
$\Phi_1(z)=(P+ z P^\perp) U^*$ and $\Phi_2(z)=U(P^\perp+ z P)$ are so that $\Phi_1(z) \Phi_2(z)=\Phi_2(z)\Phi_1(z)=z$ ($z \in \D$) for some unitary $U\in\clb(\cle)$ and projection $P\in\clb(\cle)$.

In particular, if $N\in \clp\cll_+(T)$ is a contraction, then $(\Pi, H^2_{ \cle}(\mathbb{D}), (M^*_{\Phi_1}, M^*_{\Phi_2}))$ is a co-isometric pseudo-extension of $T$.
\end{Theorem}

\begin{proof}
Let $N\in \clp\cll_+(T)$. Set $P_T:=T_1T_2$. Since $P_T^*NP_T\le N$, we set
$R^2:=N-P_T^*NP_T$ and $\clr=\overline{\text{Ran}}\, R$. Also, since $T_i^*NT_i\le N$ for $i=1,2$, we set $R_i^2:= N-T_i^*NT_i$ and $\clr_i:=\overline{\text{Ran}}\, R_i$ for $i=1,2$. Consider the map $\pi:\clh\to H^2_{\clr}(\D)$, defined by $h\mapsto \sum_{k\in \Z_+}z^k RP_T^k h$. A simple calculation using $P_T^{* k}NP_T^k\to 0$ in SOT reveals that 
\[
\sum_{k\in \Z_+} P_T^{* k}R^2P_T^k=\sum_{k\in \Z_+}P_T^{* k}(N-P_T^* N P_T)P_T^k =N,
\]
where the above sum converges in SOT. This in particular shows that $\pi$ is a well-defined bounded operator and  $\pi^*\pi=N$. We now proceed to find several unitaries which are needed to construct $\Phi_1$ and $\Phi_2$. First, using the operator identity 
\begin{equation}\label{operator identity}
R^2=R^2_1 + T^*_1 R^2_2 T_1=  T_2^*R_1^2T_2+ R_2^2,
\end{equation}
we get an isometry 
\begin{align*}
U: \{R_1 T_2 h \oplus R_2 h : h\in \clh\} \to \{R_1 h \oplus R_2 T_1 h : h \in \clh\},
\end{align*}
defined by 
\begin{align}\label{pair-U1}
    U(R_1  T_2 h, R_2 h)=( R_1 h, R_2 T_1 h) \quad (h \in \clh).
\end{align}
Then, by adding a Hilbert space $\cld$ if necessary, we extend the isometry $U$ to a unitary, again denoted by $U$, acting on  $(\cld \oplus \clr_1) \oplus \clr_2$ such that 
\[U((0_{\cld}, R_1  T_2 h), R_2 h)=((0_{\cld}, R_1 h), R_2 T_1 h) \quad ( h \in \clh).\]
We set $\cle:=\cld \oplus \clr_1 \oplus \clr_2$ and again using ~\eqref{operator identity}, we get an isometry $V:\clr \to \cle $, defined by 
\[V (Rh) =((0_{\cld}, R_1 h),  R_2 T_1 h) \quad (h\in \clh).\]
Finally, with the help of the embedding 
 $\iota_1: \cld \oplus \clr_1 \to \cle$, and  $\iota_2: \clr_2 \to \cle$, defined by 
 \[\iota_1 (d, h_1)=(d, h_1, 0) \quad \text{and}\quad  \iota_2  h_2=(0,0, h_2) \quad(d\in \cld, h_1\in \clr_1, h_2 \in \clr_2),\]
 and the orthogonal projection $P: \cle \to \cle$, defined by
 \[P(d, h_1, h_2)=(0, 0, h_2), \quad (d\in \cld, h_1\in \clr_1, h_2 \in \clr_2),\]
we define 
\[W_1: \cle \oplus (\cld \oplus \clr_1) \to \cle \oplus (\cld \oplus \clr_1),\]
defined by 
\[W_1=
\begin{bmatrix}
U & 0\\
0 & I
\end{bmatrix} 
\begin{bmatrix}
P & \iota_1\\
\iota^*_1 & 0
\end{bmatrix}.\]
Then $W_1$ is a unitary and by a straightforward computation, we have
\[ W_1 \begin{pmatrix}
    V R h\\
    (0_{\cld}, R_1 P_T h)
    \end{pmatrix}= \begin{pmatrix}
    V R T_1 h \\
    (0_{\cld}, R_1 h)
    \end{pmatrix} \quad (h\in\clh).\]
Thus, $W_1$ satisfies the hypothesis of Lemma~\ref{sb3}
with $\sq=P_T, \sr=R, \s= R_1, \st= T_1, \tilde{\cle}=\cld$,
and $\sv=V$, and therefore, 
\[
\Pi T_1 = M^* _{\Phi_1} \Pi,
\]
where $\Pi=(I_{H^2(\D)}\otimes V)\pi$ and $\Phi_1(z)=(P+ z P^\perp) U^*$ $(z\in\D)$ is the transfer function of $W_1^*$.
Similarly, if we consider
\[W_2: \cle \oplus \clr_2 \to \cle \oplus \clr_2,\]
defined by 
\[W_2=
\begin{bmatrix}
P^\perp & \iota_2\\
\iota^*_2 & 0
\end{bmatrix} 
\begin{bmatrix}
U^* & 0\\
0 & I
\end{bmatrix},\]
then $W_2$ is a unitary so that 
\[ W_2 \begin{pmatrix}
    V R h\\
     R_2 P_T h
    \end{pmatrix}= \begin{pmatrix}
    V R T_2 h \\
     R_2 h
    \end{pmatrix}\quad (h\in\clh).\] 
Then $W_2$ also satisfies the hypothesis of Lemma~\ref{sb3} with $\sq=P_T, \sr=R, \s= R_2, \st= T_2$, $\tilde{\cle}=0$
and $\sv=V$.
Therefore,  we have
\[\Pi T_2 = M^* _{\Phi_2} \Pi,\]
where $\Phi_2(z)=U(P^\perp+ z P)$ ($z \in \D$) is the transfer function corresponding to $W_2^*$. Now, it easy to check that $(M_{\Phi_1}, M_{\Phi_2})\in \clb(H^2_{ \cle}(\mathbb{D}))^2$ is a pair of commuting isometries,  $N=\Pi^* \Pi$, and $\Phi_1(z) \Phi_2(z)=\Phi_2(z)\Phi_1(z)=z$ ($z \in \D$). This proves the `only if' part of the theorem. 

On the other hand, since $M_{\Phi_1} M_{\Phi_2}=M_z$, then $\Pi P_T=M_z^*\Pi$, and therefore, 
\[P_T^{* k}N P_T^k= P_T^{* k}\Pi^*\Pi P_T^k=\Pi^* M_z^k M_z^{* k}\Pi\to 0\ \text{ in SOT as } k\to \infty.\] 
The converse part now follows. Finally, if $N\in \clp\cll_+(T)$ is a contraction, then $\pi$ is a contraction, and thus, $\Pi$ is a contraction. This completes the proof.
\end{proof}

Next we prove Theorem~\ref{M_theta}, and for the convenience of the reader we state the theorem again here. The following set of notations will be used in the proof.  For an $n$-tuple of commuting contractions $T=(T_1,\dots, T_n)$ on $\clh$ and $i=1,\ldots,n$, we denote by $\tilde{T_i}$ the $(n-1)$-tuple obtained from $T$ by deleting $T_i$. That is, 
\[
\tilde{T}_i=(T_1,\dots, T_{i-1}, T_{i+1},\dots, T_n).
\]
Also, we use the notation $P_{\tilde{T_i}}:= T_1\cdots T_{i-1}T_{i+1}\cdots T_n$.
\begin{Theorem}\label{Re: M_theta}
Let $T=(T_1, \ldots,T_n)\in \clb(\clh)^n$ be an $n$-tuple of commuting contractions with $n>2$. Then $N\in \clp \cll_+(T)$ if and only if there exist a Hilbert space $\cle$, a bounded operator $\Pi:\clh \to H^2_{ \cle}(\mathbb{D})$, and an $n$-tuple of isometries $M=(M_{\Theta_1}, \ldots, M_{\Theta_n})$ on $H^2_{ \cle}(\mathbb{D})$  such that 
\[N=\Pi^* \Pi\, \text{and } \Pi T_i= M^*_{\Theta_i}\Pi  
  \quad (i=1,\ldots,n),\]
where for each $i=1,\ldots,n$, 
$\Theta_i(z)=(P_i+zP_i^{\perp})U_i^*$ for some unitary $U_i\in \clb(\cle)$ and orthogonal projection $P_i\in\clb(\cle)$.
Moreover, 
\[
M^*_{z} \Pi=M^*_{\Theta_1}M^*_{\Theta_2} \dots M^*_{\Theta_n}\Pi.
\]

In particular, if $N\in \clp \cll_+(T)$ is a contraction, then $(\Pi, H^2_{ \cle}(\mathbb{D}), (M^*_{\Theta_1},\ldots, M^*_{\Theta_n}))$ is a non-commuting co-isometric pseudo-extension of $T$.
\end{Theorem}

\begin{proof}
Let $N\in \clp\cll_+(T)$. Set $P_T:=T_1\cdots T_n$.
Consider the map $\pi:\clh\to H^2_{\clr}(\D)$, defined by $h\mapsto \sum_{k\in \Z_+}z^k RP_T^k h$, where 
$R^2:=N-P_T^*NP_T$ and $\clr=\overline{\text{Ran}}\, R$. Then by the same argument as given in the proof of Theorem~\ref{Re: 2_M_theta}, $\pi$ is a bounded operator and $\pi^*\pi=N$.
Since $T_i^*NT_i\le N$ for $i=1,\ldots n$, we set
\begin{align*}
    R^2_i:=N-T_i^*N T_i, \quad \clr_i:=\overline{\text{Ran}}\, R_i, \quad \tilde{R}^2_i:=N-P^*_{\Tilde{T_i}} N P_{\Tilde{T_i}},  \text{ and } \tilde{\clr_i}:=\overline{\text{Ran}}\,\tilde{R}_i.
\end{align*}
Using the operator identity 
 \begin{align}\label{Id_for_d}
     P_{\Tilde{T_i}}^* R^2_i P_{\Tilde{T_i}} + \tilde{R}^2_i= R^2_i+T^*_i \tilde{R}^2_i T_i=R^2, 
 \end{align}
for each $i=1,\dots,n$, we obtain an isometry 
\begin{align*}
\Tilde{U}_i: \{R_i P_{\Tilde{T_i}} h \oplus \tilde{R}_i h : h\in \clh\} \to \{R_i h \oplus \tilde{R}_i T_i h : h \in \clh\},
\end{align*}
defined by 
\begin{align}\label{U1}
    \Tilde{U}_i(R_i  P_{\Tilde{T_i}} h, \tilde{R}_i h)=( R_i h, \tilde{R}_i T_i h) \quad (h \in \clh).
\end{align}
Then, for $i=1, \ldots,n$, by adding a Hilbert space $\cld_i$ if necessary, we extend the isometry $\Tilde{U}_i$ to a unitary, again denoted by $\Tilde{U}_i$, acting on  $\cld_i \oplus \clr_i \oplus \tilde{\clr}_i$ satisfying 
\[\Tilde{U}_i(0_{\cld_i}, R_i  P_{\Tilde{T_i}} h, \tilde{R}_i h)=(0_{\cld_i}, R_i h, \tilde{R}_i T_i h) \quad ( h \in \clh).\]
Again, for each $i=1,\ldots,n$, using the identity \eqref{Id_for_d}, we get an isometry $V_i:\clr \to \cld_i \oplus \clr_i \oplus \tilde{\clr}_i $, defined by 
\[V_i (Rh) =(0_{\cld_i}, R_i h,  \tilde{R}_i T_i h) \quad (h\in \clh).\]
Now, given the isometries $V_i$'s, it is always possible to find Hilbert spaces $\cld$ and $\cle_1, \ldots, \cle_n$ such that each $V_i$ extends to a unitary $\Tilde{V}_i:\clr \oplus \cld \to \cld_i \oplus \cle_i \oplus \clr_i \oplus \tilde{\clr}_i$. Set $\cls_i:= \cld_i \oplus \cle_i\oplus \clr_i\oplus \tilde{\clr}_i$. Finally, we extend each $\Tilde{U}_i$ to a unitary on $\cls_i$, again denoted by $\Tilde{U}_i$, such that \[\Tilde{U}_i(0_{\cld_i}, 0_{\cle_i}, R_i  P_{\Tilde{T_i}} h, \tilde{R}_i h)=(0_{\cld_i}, 0_{\cle_i}, R_i h, \tilde{R}_i T_i h) \quad ( h \in \clh, i=1,\dots,n).\]
We now have all the ingredients necessary to construct the lifting of $T_i$ as we have done in the proof of Theorem~\ref{Re: 2_M_theta}.   For $i=1,\ldots,n$, with the help of the embedding 
 $\iota_i: \cld_i \oplus \cle_i\oplus \clr_i \to \cls_i$, defined by 
 \[\iota_i (d, h, k)=(d, h, k, 0) \quad (d\in \cld_i, h\in \cle_i, k\in \clr_i),\]
 and the orthogonal projection $\Tilde{P}_i: \cls_i \to \cls_i$, defined by
 \[\Tilde{P}_i(d, h, k, r)=(0, 0, 0, r) \quad (d \in\cld_i, h \in \cle_i, k \in \clr_i, r \in \tilde{\clr}_i),\]
 we define a unitray 
\[W_i: \cls_i \oplus (\cld_i \oplus \cle_i \oplus \clr_i) \to \cls_i \oplus (\cld_i \oplus \cle_i \oplus \clr_i),\]
such that its block matrix representation with respect to the above decomposition is given by 
\[W_i=
\begin{bmatrix}
\Tilde{U}_i & 0\\
0 & I
\end{bmatrix} 
\begin{bmatrix}
\Tilde{P}_i & \iota_i\\
\iota^*_i & 0
\end{bmatrix}.\]
 Now, by identifying $\clr$ with $\clr\oplus \{0\}\subseteq \clr\oplus\cld$, we have for $h\in \clh$ and $i=1,\dots,n$, 
\begin{align*}
    W_i \begin{pmatrix}
    \Tilde{V}_i R h\\
    (0_{\cld_i \oplus \cle_i}, R_i P_T h)
    \end{pmatrix}
    &= \begin{bmatrix}
    \Tilde{U}_i \Tilde{P}_i & \Tilde{U}_i \iota_i\\
    \iota^*_i & 0
    \end{bmatrix} \begin{pmatrix}
    V_i R h\\
    (0_{\cld_i \oplus \cle_i}, R_i P_T h)
    \end{pmatrix}\\
    &=\begin{bmatrix}
    \Tilde{U}_i \Tilde{P}_i & \Tilde{U}_i \iota_i\\
    \iota^*_i & 0
    \end{bmatrix} \begin{pmatrix}
    (0_{\cld_i \oplus \cle_i}, R_i h, \tilde{R}_i T_i h)\\
    (0_{\cld_i \oplus \cle_i}, R_i P_T h)
    \end{pmatrix}\\
   &= \begin{pmatrix}
    \Tilde{U}_i(0_{\cld_i \oplus \cle_i}, R_i P_T h, \tilde{R}_i T_i h )\\
    (0_{\cld_i \oplus \cle_i}, R_i h)
    \end{pmatrix}\\
    &= \begin{pmatrix}
    \Tilde{U}_i(0_{\cld_i \oplus \cle_i}, R_i \tilde{T}_i T_i h, \tilde{R}_i T_i h )\\
    (0_{\cld_i \oplus \cle_i}, R_i h)
    \end{pmatrix}\\
    &= \begin{pmatrix}
    (0_{\cld_i \oplus \cle_i}, R_i T_i h, \tilde{R}_i T^2_i h )\\
    (0_{\cld_i \oplus \cle_i}, R_i h)
    \end{pmatrix}\\
    &= \begin{pmatrix}
    \Tilde{V}_i R T_i h \\
    (0_{\cld_i \oplus \cle_i}, R_i h)
    \end{pmatrix}.
\end{align*}
Thus each $W_i$ satisfies the hypothesis of Lemma~\ref{sb3}
with $\sq=P_T, \sr=R, \s= R_i, \st= T_i, \tilde{\cle}=\cld_i\oplus \cle_i$,
and $\sv=\Tilde{V}_i|_{\clr}$, and thus, we have
\begin{equation}\label{inter identity}
\Tilde{\Pi} T_i = M^* _{\Phi_i} \Tilde{\Pi} \quad (i=1,\ldots,n),\end{equation}
where $\Phi_i(z)=(\Tilde{P}_i+ z \Tilde{P}_i^\perp) \Tilde{U}_i^*$ for $z \in \D$ and $\Tilde{\Pi}= (I_{H^2(\D)}\otimes \Tilde{V}_i)\pi$.
For each $i=1,\ldots,n$, we set $P_i:=\Tilde{V}^*_i \Tilde{P}_i \Tilde{V}_i$ and $U_i:=\Tilde{V}^*_i \Tilde{U}_i \Tilde{V}_i$. Then $P_i$ and $U_i$ are the orthogonal projection and unitary on $\cle$, where $\cle:=\clr \oplus \cld$. Again, by the identification of $\clr$ with $\clr\oplus \{0\}\subseteq \cle$, we can view the map $\pi$ ($=:\Pi$) from $\clh$ to $ H^2_\cle(\D)$ satisfying $\Pi^* \Pi=N$, and letting, for $i=1,\ldots,n$, $\Theta_i(z):=\Tilde{V}^*_i \Phi_i(z) \Tilde{V}_i=(P_i +z P^\perp_i)U^*_i$, 
we see that the intertwining identity \eqref{inter identity} now becomes 
\[\Pi T_i =M^*_{\Theta_i} \Pi\quad (i=1,\dots,n).\]
Moreover,
\begin{align*}
   M^*_z \Pi =\Pi P_T &= \Pi T_1 T_2 \dots T_n
    = M^*_{\Theta_1} \dots M^*_{\Theta_n} \Pi.
\end{align*}
This proves one direction of the theorem.
The other direction follows from the following:
for $i=1,\ldots,n$,
\begin{align*}
    T^*_i N T_i &= T^*_i \Pi^* \Pi T_i
    = \Pi^* M_{\Theta_i} M^*_{\Theta_i} \Pi
    \leq \Pi^* \Pi=N,
\end{align*}
and 
 for $h \in \clh$, 
\begin{align*}
    \|P_T^{* k} N P_T^k h\|&= \|P_T^{* k} \Pi^* \Pi P_T^k h\|
    = \|\Pi^* M^k _z M^{* k}_z \Pi h\| \to 0 \quad \text{as } k\to \infty.
\end{align*}
In the case, if $N\in \clp \cll_+(T)$ is a contraction, then $\Pi$ is a contraction. This completes the proof.
\end{proof}
It should be noted that the tuple of isometries $(M_{\Theta_1},\dots, M_{\Theta_n})$ do not commute, in general and which is the main difference with Theorem~\ref{Re: 2_M_theta}. However, as stated in the introduction, if we assume that $P_{T}^{* k}\to 0$ in SOT, then $I\in \clp\cll_+(T^*)$ and in this case the above theorem gives an explicit construction of Halmos dilation of $T$. Such a dilation result also obtained earlier in ~\cite{BS} using so called fundamental operators.


We illustrate Theorem \ref{Re: 2_M_theta} by the following example.
\begin{Example}
For $n=2$, we consider a pair of commuting matrices $T=(T_1, T_2)$ such that for $0<a<1$,
\[T_1=\begin{bmatrix}
    a & 0 & 0\\
    0& a & 0\\
    0& 0& 0
\end{bmatrix} \, \text{and }\, T_2=\begin{bmatrix}
    0 & 0 & 0\\
    0& 1 & 0\\
    0& 0& 1
\end{bmatrix}.\]
Then it is easy to verify that for $r, s, t\geq 0$, the matrix $N=\begin{bmatrix}
    r & 0 & 0\\
    0& s & 0\\
    0& 0& t
\end{bmatrix}$ is in $\clp\cll_+(T)$.
Consider the bounded operator $\Pi:\mathbb C^3 \to H^2_{\mathbb C^4}(\D)$, defined by
 \[(\Pi h)(z)=(\sqrt{r (1-a^2)} h_1, \sqrt{s (1-a^2)} h_2, \sqrt{t} h_3, a\sqrt{r} h_1)+ \sum_{k\geq 1}^\infty z^k (0, a^k \sqrt{s (1-a^2)} h_2, 0,0)\]
 for $z\in\D$ and $ h=(h_1, h_2, h_3)\in \mathbb C^3$, any unitary $U:\mathbb C^4 \to \mathbb C^4$ satisfying 
 \[U(0, \sqrt{s (1-a^2)} h_2, \sqrt{t} h_3, \sqrt{r} h_1)=(\sqrt{r (1-a^2)} h_1, \sqrt{s (1-a^2)} h_2, \sqrt{t} h_3, a \sqrt{r} h_1 ), \]
and 
the projection $P:\mathbb C^4 \to \mathbb C^4$, defined by $P(h_1, h_2, h_3, h_4)=(0,0,0, h_4)$ for $(h_1, h_2, h_3, h_4)\in \mathbb C^4$. Then, it is straight forward to check that 
\[N=\Pi^* \Pi,
\quad \text{and}\quad \Pi T_i= M^*_{\Phi_i}\Pi \quad(i=1,2),\]
where $\Phi_1(z)=(P+ z P^\perp) U^*$ and $\Phi_2(z)=U(P^\perp+ z P)$.  
\end{Example}

\section{$(S, T)$-Toeplitz operators }\label{Sec4}  
  
For $n$-tuples of commuting contractions $T=(T_1, \ldots,T_n)\in \clb(\clh)^n$ and $S=(S_1,\ldots,S_n)\in \clb(\clk)^n$, we study solutions of the following operator equations in this section:
\begin{align}\label{ME1}
    S^*_i X T_i=X\quad (i=1,\ldots,n),
\end{align}
where $X\in \clb(\clh,\clk)$. The solution space of the above operator equations \eqref{ME1} is denoted by $\T(S, T)$ and an element $X$ in $\T(S, T)$ is called an $(S, T)$-Toeplitz operator.
We first describe the canonical construction of $(S,T)$-Toeplitz operator corresponding to each elements in $\clb(\clh,\clk)$ using Banach limits. Let us recall few elementary facts about Banach limits.    
Consider the Banach space $l^{\infty}(\Z^n_+)$, defined by 
\[l^\infty(\Z^n_+):=\{x=\{x_\alpha\}_{\alpha \in \Z^n_+}: x_\alpha \in \mathbb{C}, \|x\|=\sup_{\alpha \in \Z^n_+} |x_\alpha| < \infty \}.\]
For $x=\{x_\alpha\}_{\alpha \in \Z^n_+} \in l^\infty(\Z^n_+)$, we set for each $i=1,\ldots,n$, $x^{(i)}:=\{x_{\alpha+e_i}\}_{\alpha \in \Z^n_+}$, where $e_i=(0,\ldots,0,\underbrace{1}_{i\text{th place}},0,\ldots,0) \in \Z_+^n$  and define a closed subspace $M_i$ of $l^\infty(\Z^n_+)$ as 
\[M_i:=\{x-x^{(i)}: x=\{x_\alpha\}_{\alpha \in \Z^n_+} \in l^\infty(\Z^n_+) \}.\]
Banach limits in our context will be a subset of the unit ball of the dual space $(l^\infty(\Z^n_+))^*$ denoted by $\clm(\Z^n_+)$ and consists of $\phi \in (l^\infty(\Z^n_+))^*$ such that $\phi(M)=0$ and $\phi(\bm{1})=1$, where $M=\cap_1^n M_i$ and $\bm{1}:=\{1, 1,\ldots,\}$. Note that $M$ is a non-trivial subspace of $l^\infty(\Z^n_+)$ as $\{1, 0, \ldots,\} \in M $, and it is easy to see that $\bm{1} \in l^\infty(\Z^n_+) \setminus M$. 


Let $T=(T_1,\ldots,T_n) \in \clb(\clh)^n$ and $S=(S_1,\ldots,S_n)\in \clb(\clk)^n$ be $n$-tuples of commuting contractions.
 For $X\in \clb(\clh, \clk)$ and for each $(h,k)\in \clh\times\clk$, 
 the net $\{\langle S^{* \alpha} X T^\alpha h, k \rangle\}_{\alpha \in \Z^n_+ } \in l^\infty(\Z^n_+)$. Now, for each $\phi \in \clm(\Z^n_+)$,
we define a sesquilinear form $B^X_\phi : \clh \times \clk \to \mathbb{C}$ by
\[B^X_\phi(h, k)=\phi (\{\langle S^{* \alpha} X T^\alpha h, k \rangle\}_{\alpha \in \Z^n_+ }) \quad (h\in\clh , k \in \clk).\]
A simple estimate shows that $B^X_{\phi}$ is a bounded sesquilinear form and $\|B^X_{\phi}\|\le \|\phi\|\|X\|\le \|X\|$. 
Hence, by the Riesz representation theorem, there exists $Y_\phi(X) \in \clb(\clh, \clk)$ such that for $h\in \clh, k \in \clk$, $B^X_\phi (h, k)=\langle Y_\phi (X) h, k \rangle$.
Also note that, for $h\in \clh, k \in \clk$, and $i=1,\ldots,n$, 
\begin{align*}
    \langle S^*_i Y_\phi(X) T_i h, k \rangle 
    &=B^X_\phi (T_i h, S_i k)\\
    &=\phi (\{\langle S^{* \alpha} X T^\alpha  T_i h, S_i k \rangle\}_{\alpha \in \Z^n_+ })\\
    &=\phi (\{\langle S^{* (\alpha +e_i)} X T^{(\alpha +e_i)} h, k \rangle\}_{\alpha \in \Z^n_+ })\\
    &= \phi (\{\langle S^{* \alpha} XT^\alpha h, k \rangle\}_{\alpha \in \Z^n_+ }),  \quad \text{as} \quad \phi(M)=0,\\
    &=\langle Y_\phi(X) h, k \rangle.
\end{align*}
Hence, $Y_\phi(X)\in \T(S, T)$. By the above analysis, corresponding to each $\phi\in \clm(\Z^n_+)$, the map \[Y_\phi :\clb(\clh, \clk) \to \T(S, T),\ X\mapsto Y_\phi(X),\] 
defines a bounded linear operator. Also since any bounded linear functional of norm one is completely contractive and the map $X\in \clb(\clh,\clk)\mapsto S^{*\alpha}XT^{\alpha}\in\clb(\clh,\clk)$ is completely contractive for any $\alpha\in \Z_+^n$, it follows that $Y_{\phi}$ is in fact a completely contractive map.
By the construction it also follows that for $X\in \T(S, T)$, $Y_\phi(X)=X$.  

Thus we have proved the following theorem. For the base case $n=1$, the result is obtained in ~\cite{GC}. 
\begin{Theorem}\label{B_limt}
Let $T=(T_1,\ldots,T_n) \in \clb(\clh)^n$ and $S=(S_1,\ldots,S_n)\in \clb(\clk)^n$ be $n$-tuples of commuting contractions. Then for $X\in \clb(\clh, \clk)$ and $\phi \in \clm(\Z^n_+)$, $Y_\phi(X)\in \T(S, T)$. For each $\phi \in \clm(\Z^n_+)$, the map 
\[Y_\phi :\clb(\clh, \clk) \to \T(S, T), \ X\mapsto Y_\phi(X),\] is a surjective completely contractive map. Moreover, if $X\in \T(S, T)$, then $Y_\phi(X)=X$ for all $\phi \in \clm(\Z^n_+)$.   
\end{Theorem}

By the above theorem, one concludes that $\T(S,T)$ is non-trivial if and only if the range of the map $Y_{\phi}$ is not $\{0\}$ for any $\phi\in \clm(\Z^n_+)$. However, it is difficult to check whether $Y_{\phi}$ is $\{0\}$ or not, in general. Below, we provide two necessary conditions in terms of isometric and unitary pseudo-extensions for $\T(S, T)$ to be non-trivial. 
Here, we use the terminology that two pseudo-extensions $(\clj, \clk,U)$ and $(\Tilde{\clj}, \Tilde{\clk},\Tilde{U})$ of $T$ are unitarily equivalent if there exists a unitary $W:\clk \to \Tilde{\clk}$ such that 
\[W U_i=\Tilde{U}_i W, \, i=1, \ldots,n, \, \text{and}\,\ W \clj=\Tilde{\clj}.\]

\begin{Proposition}\label{Th2}

 Let  $S=(S_1,\ldots,S_n)\in \clb(\clk)^n$ and $T=(T_1, \ldots,T_n)\in \clb(\clh)^n$ be $n$-tuples of commuting contractions. If $\T(S, T)\neq\{0\}$, then 
 \item
 (a) Adjoint of the product contractions $P_T$ and $P_S$ are not pure, where $P_T=T_1\cdots T_n$ and $P_S=S_1\cdots S_n$. 
 \item
   (b) $S$ and $T$ have canonical isometry pseudo-extensions. 
 \item
   (c) $S$ and $T$ have unique (up to unitary equivalence) canonical unitary pseudo-extensions.
 \end{Proposition}

\begin{proof}[Proof of part (a):]

We prove by the contradictory method. Let us assume that the adjoint of the product contraction $P_T$ is pure. Let $X$ be a non-zero operator such that $X\in \T(S,T)$. Since for $i=1,\ldots,n$, $S_i^* X T_i= X$, we also have 
\begin{align} \label{pure_1}
    P_S^{* k}X P_T^k=X \quad (k\in \Z_+).
\end{align}
Then for $h \in \clh$, 
\begin{align*}
    \|X h\|&=\|P_S^{*k} X P_T^k h\|
     \leq \|P_S^{* k}\| \|X\| \|P_T^k h\|
    \leq \|X\| \|P_T^k h\| 
     \to 0 \quad \text{as }k\to \infty. 
\end{align*}
Therefore, $X=0$, which contradicts that $X$ is a non zero operator. 
 Taking adjoint of the equation \eqref{pure_1}, we can similarly prove that the adjoint of $P_S$ is not a pure contraction.
 

 
 \noindent\textit{Proof of part (b)}: Since the adjoint of product contractions $P_T$ and $P_S$ are not pure. Then the proof follows from Theorem~\ref{non-zero Toeplitz}. 
 
 \noindent\textit{Proof of part (c)}: We only show that $T$ has a canonical unitary pseudo-extension as the arguments for $S$ is identical. Since the adjoint of $P_T$ is not pure, let $(\clj_T, \clq_T, V)$ be a canonical isometric pseudo-extension of $T$ as constructed in the proof of Theorem~\ref{non-zero Toeplitz}, 
 where $V=(V_1,\ldots, V_n)$ on $\clq_T$ is a $n$-tuple of commuting isometries.  
 Suppose that $U=(U_1, \ldots,U_n)$ on $\clk_T$ is a minimal unitary extension of $V$. Since $\clq_T\subseteq \clk_T$, we view $\clj_T:\clh\to \clk_T$ and note that 
for $i=1,\ldots,n$ and $h\in \clh$, 
 \[ U_i \clj_T h=V_i \clj_T h=\clj_T T_i h,\]
 and
\[\clj_T^* \clj_T= \text{SOT}-\lim_{k \to \infty} P_T^{* k}P_T^k.\]
Therefore, $(\clj_T, \clk_T, U)$ is a canonical unitary pseudo-extension of $T$. 
For uniqueness, let us assume that $(\clj, \clk, \hat{U}=(\hat{U}_1, \ldots, \hat{U}_n))$ and $(\Tilde{\clj}, \Tilde{\clk}, \Tilde{U}=(\Tilde{U}_1, \ldots\Tilde{U}_n))$ be canonical unitary pseudo-extension of $T$. 
 Since canonical pseudo-extensions are minimal, using minimality we define an operator $W:\clk \to \Tilde{\clk}$, by
 \[f(\hat{U}, \hat{U}^*)\clj h\mapsto f(\Tilde{U}, \Tilde{U}^*)\Tilde{\clj}h\]
 for $h\in \clh$ and polynomial $f$ in $\z$ and $\overline{\z}$ ($\z=(z_1, \ldots,z_n)\in \mathbb{C}^n$). If $f=\sum a_{\alpha, \beta} \z^\alpha \overline{\z}^\beta$, then for $h\in \clh$,
\begin{align*}
    \|f(\hat{U}, \hat{U}^*)\clj h\|^2=&\sum a_{\alpha,\beta} \langle \clj^* \hat{U}^{* \beta} \hat{U}^\alpha \clj h, h \rangle\\ 
    =&\sum a_{\alpha,\beta} \langle T^{* \beta} \clj^* \clj T^\alpha h, h \rangle\\
     =&\text{SOT}-\lim_{k\to \infty}\sum a_{\alpha,\beta} \langle T^{* \beta} P_T^{* k} P^k_T T^\alpha h, h \rangle.
\end{align*}
Since the last term only depends on $T$, $W$ is a unitary. 
It is evident from the definition $W$ that $W \hat{U}_i=\Tilde{U}_i W$ ($i=1,\ldots,n$) and $W \clj=\Tilde{\clj}$. 
This completes the proof.
 \end{proof}

It turns out that none of the assertions in Proposition~\ref{Th2} is sufficient for the non-triviality of $\T(S, T)$. The following example demonstrates it.
\begin{Example}
For $n = 1$ and $\clh=\clk= H^2(\D)$, we consider $S$ as a unilateral shift ($M_z$) on $H^2(\D)$ and $T$ as an identity operator on $H^2(\D)$. It is clear that adjoint of $M_z$ and $I$ are not pure. However, in this case $\T(S,T)=\{0\}$ as any operator which intertwines a unitary and a pure contraction has to be zero.
\end{Example}

We have seen in Proposition~\ref{Th2} that the existence of a non-zero element in $\T(S,T)$ implies $S$ and $T$ have canonical isometric pseudo-extensions $(\clj_S, \clq_S, W)$ and $(\clj_T, \clq_T, V)$, respectively. In such a situation, it is natural to ask how $\T(S,T)$ and $\T(W, V)$ are related? The following result answer it immaculately. 

 \begin{Theorem}
Let $T=(T_1,\ldots,T_n) \in \clb(\clh)^n$ and $S=(S_1,\ldots,S_n)\in \clb(\clk)^n$ be $n$-tuples of commuting contractions, and let $(\clj_T, \clq_T, V)$ and  $(\clj_S, \clq_S, W)$ be canonical isometric pseudo-extensions of $T$ and $S$, respectively. Then 
$$\T(S,T)=\clj_S^* \T(W, V) \clj_T.$$ 
Moreover, for each  $A\in \T(S,T)$, there exists a $B\in \T(W, V)$ such that $A=\clj_S^* B \clj_T$ and $\|A\|=\|B\|$.
\end{Theorem} 

\begin{proof}
Let us assume that $A \in \T (W, V)$. By the hypothesis, for $i=1,\ldots,n$, $W_i^* A V_i= A, V_i \clj_T = \clj_T T_i$ and $ W_i \clj_S = \clj_S S_i$. Then for $i=1,\dots,n$,
\begin{align*}
    S_i^* \clj_S^* A \clj_T T_i &=\clj_S^* W^*_i A V_i \clj_T
    =\clj_S^* A \clj_T.
\end{align*}
Thus, $\clj_S^* A \clj_T \in \T(S,T)$, and therefore $\clj_S^* \T(W, V) \clj_T \subseteq \T(S,T)$. For the other containment, let $A \in \T(S,T)$ be such that $\|A\| \leq 1$.
If $P_S=S_1\cdots S_n$ and $P_T= T_1\cdots T_n$, then 
\[ P_S^{* k}A P_T^k=A  \quad (k\in \Z_+).\]
Consequently, for $k\in\Z_+$,
\begin{align*}
    AA^*&=P_S^{* k} A P_T^k P_T^{* k} A^* P_S^k
    \leq P_S^{* k} A A^* P_S^k
    \leq P_S^{* k} P_S^k.
\end{align*}
Since $(\clj_S, \clq_S, W)$ is a canonical isometric pseudo-extension of $S$, passing to the SOT limit in the above inequality we get $A A^*\leq \clj^*_S\clj_S$. Therefore, by the Douglas' factorization theorem, we get a contraction $C:\clh \to \clq_S$ such that $A=\clj_S^* C$. Since the isometric pseudo-extension $(\clj_S, \clq_S, W)$ of $S$ is minimal, then $\overline{\text{Ran}}\, \clj_S= \clq_S$ and $\clj_S^*$ is one-to-one map on $\clq_S$. So the identity   
\[
\clj_S^*C= A= S_i^*A T_i= S_i^*\clj_S^*C T_i= \clj_S^* W_i^*CT_i
\]
implies that $W_i^* C T_i = C$ for $i=1,\ldots, n$. Furthermore, for $k\in\Z_+$, $P_W^{ * k} C P_T^k=C$,
where $P_W=W_1\cdots W_n$. Again by a similar argument as before, we have for $k\in \Z_+$, 
\begin{align*}
    C^*C&=P_T^{* k} C^* P_W^{ k} P_W^{* k} C P_T^k
    \leq P_T^{* k} C^* C P_T^k
    \leq P_T^{* k} P_T^k.
\end{align*}
Therefore, $C^* C\leq \clj^*_T\clj_T$ as $(\clj_T, \clq_T, V)$ is a canonical pseudo-extension of $T$. Applying the Douglas' factorization theorem, we obtain a contraction $B:\clq_T \to \clq_S$ such that $C= B \clj_T$. Consequently, for each $i=1,\ldots,n$, 
\begin{align*}
    B \clj_T= C= W_i^* CT_i= W_i^* B\clj_TT_i=W_i^*BV_i \clj_T,
\end{align*}
and this implies that $W^*_i B V_i=B$ for $i=1,\ldots,n$. Thus, $A= \clj_S^* C= \clj_S^* B\clj_T$ for some $B\in\T(W,V)$ and $\|B\|\le 1$.
Hence, $\T(S,T)=\clj_S^* \T(W, V) \clj_T$.
Furthermore, the above construction shows that if $A$ is a contraction, then we can choose $B$ to be a contraction as well. This proves that for $A \in \T(S,T)$, we can find $B\in \T(W, V)$ such that $A= \clj_S^* B \clj_T$ and $\|A\|=\|B\|$. This completes the proof.
\end{proof}

It is well known that a 
tuple of commuting isometries on a Hilbert space always extends to a tuple of commuting unitaries. As a consequence any canonical isometric pseudo-extension also can be extended to a canonical unitary pseudo-extension. Then next result describes the effects on Toeplitz operators if one passes to unitary extensions.
This result can be obtained by considering the abelian semigroup generated by the tuple of isometries and the semigroup of their unitray extensions and then invoking the result of P. Muhly (\cite{PM}). However, we provide a direct proof as the set of arguments are quite standard. 

\begin{Theorem}
Let $V=(V_1,\ldots, V_n)\in \clb(\clh)^n$ and $W=(W_1,\ldots, W_n)\in \clb(\clk)^n$ be $n$-tuples of commuting isometries with minimal unitary extensions $\Tilde{V}=(\Tilde{V_1},\ldots,\Tilde{V_n})\in \clb(\tilde{\clh})^n$ and $\Tilde{W}=(\Tilde{W_1},\ldots,\Tilde{W_n})\in \clb(\tilde{\clk})^n$, respectively. Then $$\T(W,V)= P_{\clk} \T(\Tilde{W}, \Tilde{V})|_{\clh}.$$ 
Moreover, for each $A \in \T(W, V)$, there is a $B \in \T(\Tilde{W},\Tilde{V})$ such that $A=P_{\clk} B |_{\clh} $ with $\|A\|=\|B\|$. 
\end{Theorem}

\begin{proof}
Let $A\in \T(\Tilde{W},\Tilde{V})$. Then for $i=1,\dots,n$, the following chain of identities 
\begin{align*}
    W_i^* P_{\clk} AP_{\clh} V_i&=P_{\clk} \Tilde{W}_i^* P_{\clk} A \Tilde{V}_i|_{\clh}= P_{\clk} \Tilde{W}^*_i A \Tilde{V}_i|_{\clh}=P_{\clk} A|_{\clh}.
\end{align*}
shows that $\T(W, V)\supseteq P_{\clk} \T(\Tilde{W},\Tilde{V})|_{\clh}$.
For the other containment, let $ C\in \T(W, V)$. Then for $\alpha=(\alpha_1,\ldots,\alpha_n)\in \Z_+^n$,
 \[W^{* \alpha} C V^\alpha= C .\]
Now, we use a somewhat standard technique to `enlarge' $C$ and get hold of an operator $B\in\T(\Tilde{W},\Tilde{V}) $. For $\alpha=(\alpha_1,\ldots,\alpha_n)\in \Z_+^n$, we set 
\[B_\alpha: = \Tilde{W}^{* \alpha} P_{\clk} C P_{\clh} \Tilde{V}^\alpha, P_{1, \alpha}:=\Tilde{V}^{* \alpha}P_\clh \Tilde{V}^\alpha,\ \text{and  } P_{2 , \alpha}:= \Tilde{W}^{* \alpha} P_\clk \Tilde{W}^\alpha.\]
It is easy to perceive that $P_{1,\alpha}$ and $P_{2,\alpha}$ are projections onto $\Tilde{V}^{* \alpha} (\clh)$ and $\Tilde{W}^{* \alpha} (\clk)$, respectively and $\{P_{1,\alpha}\}_{\alpha\in\Z_+^n}$ and $\{P_{2,\alpha}\}_{\alpha\in\Z_+^n}$ are increasing net of projections. Moreover, $P_{1,\alpha}\uparrow I_{\tilde{\clh}}$ and $P_{2,\alpha} \uparrow I_{\tilde{\clk}}$ in SOT, which follows from the minimality of the extensions, that is, 
\[\tilde{\clh}=\bigvee_{\alpha \in \Z^n_+}\Tilde{V}^{* \alpha}(\clh)\ \text{and } \tilde{\clk}=\bigvee_{\alpha \in \Z^n_+} \Tilde{W}^{* \alpha}(\clk).\]
Also, it is easy to observe that the net$\{B_\alpha\}_{\alpha \in \Z_+^n}$ is uniformly bounded and it is bounded by $\|C\|$. Now, for $\alpha \geq \beta\geq 0$, observe that
\begin{align*}
    P_{2, \beta} B_\alpha P_{1,\beta}&=\Tilde{W}^{* \beta} P_{\clk} \Tilde{W}^\beta \Tilde{W}^{* \alpha} P_{\clk} C P_{\clh} \Tilde{V}^\alpha \Tilde{V}^{* \beta} P_\clh \Tilde{V}^\beta\\
    &=\Tilde{W}^{* \beta} P_\clk  \Tilde{W}^{* (\alpha-\beta)} P_\clk C P_\clh \Tilde{V}^{(\alpha-\beta)} P_\clh \Tilde{V}^\beta\\
    &=\Tilde{W}^{* \beta} P_\clk  W^{* (\alpha-\beta)}  C V^{(\alpha-\beta)} P_\clh \Tilde{V}^\beta\\
    &=\Tilde{W}^{* \beta} P_\clk C P_\clh \Tilde{V}^\beta\\
    &=B_\beta.
\end{align*}
Therefore, $P_{2, \beta} B_\alpha P_{1,\beta}$ is independent of $\alpha$ for $\alpha \geq \beta$. Consequently, for $x= P_{1,\beta}(x_1)\in P_{1, \beta} \tilde{\clh}$ and $y=P_{2,\beta}(y_1) \in P_{2, \beta} \tilde{\clk}$, 
\begin{align*}
    \lim_{\alpha \to \infty} \langle B_\alpha x, y \rangle &= \lim_{\alpha \to \infty} \langle P_{1,\beta} B_\alpha P_{2 ,\beta} x_1, y_1 \rangle= \langle B_\beta x_1, y_1\rangle.
\end{align*}
Since the net $\{B_\beta \}_{\beta \in \Z^n_+}$ is uniformly bounded, then it follows that $\lim_{\alpha \to \infty} \langle B_{\alpha}x, y \rangle$ exists for all $x\in \bigvee_{ \beta \in \Z^n_+} P_{1, \beta} \tilde{\clh}=\tilde{\clh}$ and $y\in \bigvee_{\beta \in \Z^n_+}P_{2, \beta} \tilde{\clk}=\tilde{\clk}$. In other words, $B_{\alpha}$ converges in the weak operator topology to $B\in \clb(\tilde{\clh}, \tilde{\clk})$, say. Now, we do a routine computation to show that 
$P_\clk B|_{\clh}=C$ and for $i=1,\ldots,n$, $\Tilde{W}^*_i B \Tilde{V}_i=B$.
To this end, for $x\in \clh$ and $y \in \clk$,
\begin{align*}
    \langle P_\clk B|_{\clh} x,y \rangle&= \lim_{\alpha \to \infty} \langle P_\clk \Tilde{W}^{* \alpha} P_\clk C P_\clh \Tilde{V}^\alpha x, y \rangle=\lim_{\alpha \to \infty} \langle  W^{* \alpha} C  V^\alpha x, y \rangle=\langle C x, y\rangle.
\end{align*}
For $\alpha=(\alpha_1,\ldots,\alpha_n)\in \Z_+^n$ and $e_i=(0,\ldots,0,\underbrace{1}_{i\text{th place}},0,\ldots,0) \in \Z_+^n$, note that  
\[\Tilde{W}^*_i B_\alpha \Tilde{V}_i=B_{\alpha +e_i}\quad (i=1,\ldots,n),\]
and therefore, for each $i=1,\ldots,n$, $\Tilde{W}^*_i B \Tilde{V}_i=B$. This completes the proof of  $\T(W, V)= P_\clk \T(\Tilde{W},\Tilde{V})|_{\clh}$. Moreover, since $B$ is the weak operator limit of $\{B_{\alpha}\}_{\alpha\in \Z_+^n}$ and the net is uniformly bounded by $\|C\|$, it follows that $\|B\|=\|C\|$. This proves the moreover part of the theorem. 
\end{proof}

We conclude the article with a necessary and sufficient condition for $\T(U, V)$ to be non-trivial, where $U$ and $V$ are $n$-tuples of commuting unitaries.


\begin{Corollary}
Let $U=(U_1,\ldots,U_n)\in \clb(\clh)^n$ and $V=(V_1,\ldots,V_n)\in \clb(\clk)^n$ be two $n$-tuples of commuting unitaries. Then $\T(U,V)\neq \{0\}$ if and only if there exist two joint reducing subspaces $\clm$ and $\cln$ of $U$ and $V$, respectively, such that 
$$(U_1 |_{\clm}, \ldots, U_n|_{\clm}) \cong (V_1 |_{\cln}, \ldots, V_n|_{\cln}).$$
\end{Corollary}

\begin{proof}

Let us assume that $\T(V,U)\neq \{0\}$. Then there is a non zero operator $B$ such that \[B U_i =V_i B \quad (i=1,\ldots,n).\]
So, if we consider $\clm=(\text{Ker}\,B)^\perp$ and $\cln=\overline{\text{Ran}}\,B$, then by ~\cite[Lemma 4.1]{DG}, we have 
$$(U_1 |_{\clm}, \ldots, U_n|_{\clm}) \cong (V_1 |_{\cln}, \ldots, V_n|_{\cln}).$$
The other implication is easy. This completes the proof.
\end{proof}

\vspace{0.1in} \noindent\textbf{Conflict of interest:}
The author states that there is no conflict of interest. No data sets were generated or analyzed during the current study.

\vspace{0.1in} \noindent\textbf{Acknowledgement:}
The author would like to thank his supervisor, Prof. Bata Krishna Das, for introducing him to the problem. The author is also grateful to him for reading the article with his prodigious patience.

\bibliographystyle{amsplain}


\end{document}